\documentclass[12pt]{amsart}
\usepackage{amsmath,amssymb,amsbsy,amsfonts,amsthm,latexsym,amsopn,amstext,
                                                    amsxtra,euscript,amscd}
\usepackage[english]{babel}
\usepackage{hyperref}
\begin{document}

\newtheorem{thm}{Theorem}
\newtheorem{lem}[thm]{Lemma}
\newtheorem{claim}[thm]{Claim}
\newtheorem{remark}[thm]{Remark}
\newtheorem{cor}[thm]{Corollary}
\newtheorem{prop}[thm]{Proposition} 
\newtheorem{definition}{Definition}
\newtheorem{question}[thm]{Open Question}
\newtheorem{conj}[thm]{Conjecture}
\newtheorem{prob}{Problem}
\def\vol {{\mathrm{vol\,}}}
\def\squareforqed{\hbox{\rlap{$\sqcap$}$\sqcup$}}
\def\qed{\ifmmode\squareforqed\else{\unskip\nobreak\hfil
\penalty50\hskip1em\null\nobreak\hfil\squareforqed
\parfillskip=0pt\finalhyphendemerits=0\endgraf}\fi}

\def\cA{{\mathcal A}}
\def\cB{{\mathcal B}}
\def\cC{{\mathcal C}}
\def\cD{{\mathcal D}}
\def\cE{{\mathcal E}}
\def\cF{{\mathcal F}}
\def\cG{{\mathcal G}}
\def\cH{{\mathcal H}}
\def\cI{{\mathcal I}}
\def\cJ{{\mathcal J}}
\def\cK{{\mathcal K}}
\def\cL{{\mathcal L}}
\def\cM{{\mathcal M}}
\def\cN{{\mathcal N}}
\def\cO{{\mathcal O}}
\def\cP{{\mathcal P}}
\def\cQ{{\mathcal Q}}
\def\cR{{\mathcal R}}
\def\cS{{\mathcal S}}
\def\cT{{\mathcal T}}
\def\cU{{\mathcal U}}
\def\cV{{\mathcal V}}
\def\cW{{\mathcal W}}
\def\cX{{\mathcal X}}
\def\cY{{\mathcal Y}}
\def\cZ{{\mathcal Z}}

\def\NmQR{N(m;Q,R)}
\def\VmQR{\cV(m;Q,R)}

\def\Xm{\cX_m}

\def \A {{\mathbb A}}
\def \C {{\mathbb C}}
\def \F {{\mathbb F}}
\def \L {{\mathbb L}}
\def \K {{\mathbb K}}
\def \N {{\mathbb N}}
\def \Q {{\mathbb Q}}
\def \R {{\mathbb R}}
\def \Z {{\mathbb Z}}
\def \fS{\mathfrak S}

\def\\{\cr}
\def\({\left(}
\def\){\right)}
\def\fl#1{\left\lfloor#1\right\rfloor}
\def\rf#1{\left\lceil#1\right\rceil}

\def\Tr{{\mathrm{Tr}}}
\def\Im{{\mathrm{Im}}}
\def\ov#1{{\overline#1}}

\def \bFp {\overline \F_p}

\newcommand{\pfrac}[2]{{\left(\frac{#1}{#2}\right)}}

\def \Prob{{\mathrm {}}}
\def\e{\mathbf{e}}
\def\ep{{\mathbf{\,e}}_p}
\def\epp{{\mathbf{\,e}}_{p^2}}
\def\em{{\mathbf{\,e}}_m}

\def\Res{\mathrm{Res}}
\def\Orb{\mathrm{Orb}}

\def\vec#1{\mathbf{#1}}
\def\flp#1{{\left\langle#1\right\rangle}_p}

\def\Crk{\mathrm{Crk}\,}

\def\mand{\qquad\mbox{and}\qquad}

\newcommand{\comm}[1]{\marginpar{%
\vskip-\baselineskip %raise the marginpar a bit
\raggedright\footnotesize
\itshape\hrule\smallskip#1\par\smallskip\hrule}}

\title{Polynomial Values in 
Affine Subspaces of Finite Fields}

\author[A. Ostafe] {Alina Ostafe} 
\address{School of Mathematics and Statistics, University of New South Wales, 
Sydney, NSW 2052, Australia}
\email{alina.ostafe@unsw.edu.au}

\date{\today}

%\begin{abstract}  In this paper we obtain new results about polynomial values falling in affine subspaces of finite fields, improving and generalising previous known results, but also extending the range of such results to polynomials of degrees higher than the characteristic of the field. Such results have immediate consequences to intersection of orbits with affine subspaces, but also for the Waring's problem in affine subspaces. We also obtain estimates for a certain new type of exponential sums. 
%\end{abstract}

\begin{abstract}  In this paper we introduce a new approach and  obtain new results 
for the problem of studying  polynomial images of  affine subspaces of finite fields. 
We  improve and generalise several previous known results,  and also extend the range of such results to polynomials of degrees higher than the characteristic of the field. Such results have 
a wide scope of applications similar to those associated with their counterparts
studying consecutive intervals over prime fields instead of  affine subspaces. 
Here we give only two
immediate consequences: to a bound on the size of the  intersection of orbits 
of polynomial dynamical systems with affine subspaces and to the Waring problem
 in affine subspaces. These results are based on  estimates for a certain new type of 
 exponential sums. 
\end{abstract}

%% \paragraph{Mathematical Subject Classification:}  


\maketitle

%---------------------------------------------------------------------
\section{Introduction}

\subsection*{Motivation} 

Given a polynomial $f$ over a field $\F$ and two ``interesting''
finite sets $\cA, \cB \subseteq \F$ it is natural to ask about the 
size of the intersection 
$$
f(\cA) \cap \cB = \# \{f(a)~:~  a \in \cA\ \text{and}\ f(a) \in \cB\}
$$
and in particular improve the trivial bound $\min\{\#\cA, \# \cB\}$
on the size of this intersection. In particular,  for the case of 
prime finite fields $\F_p$ of $p$ elements with  $\cA, \cB$ chosen as intervals of consecutive 
elements (in a natural ordering 
of elements of $\F_p$)  a series of such results have been obtained 
in~\cite{Chang3,CCGHSZ,CGOS} where also a broad variety 
 of application has been given.  For example, one 
 of motivating applications for these results comes from 
 the study of the number of points in polynomial orbits that fall
 in a given  interval, see~\cite{Chang1,Chang2,CGOS,GuShp}.

Here we mostly concentrate on the case of finite fields that are  
high degree extensions of prime fields. Furthermore,  our sets 
$\cA, \cB$ are affine subspaces which are   natural analogues 
of   intervals in these settings. 

More precisely, for a prime power  $q$ and an integer $r>1$ we
denote by $\K = \F_q$ and $\L = \F_{q^r}$  the finite fields  of 
$q$ and $q^r$ elements, respectively, and consider affine subspaces
of $\L$ over $\K$. We are especially interested in the case when 
the dimension $s$ of these spaces is small compared to $r$ and thus
standard approaches via algebraic geometry methods (such as the 
Weil bound) do not apply.

We note that a similar point of view has recently been accepted by Cilleruelo 
and Shparlinski~\cite{CillShp} and 
by Roche-Newton and Shparlinski~\cite{RNS} who obtained several results
in this direction via the methods of additive combinatorics. In fact, 
the results and method of~\cite{CillShp} apply only to a very special class 
of affine spaces, while~\cite{RNS} addresses the case of arbitrary 
affine spaces. 
Here, using a different approach, we improve some of the results of~\cite{RNS}
and also obtain a series of other results. In particular, we obtain 
some nontrivial results for a class of polynomials of degree $d \ge p$, 
where $p$ is the characteristic of $\L$, while for the inductive
method of~\cite{RNS} the condition $d< p$ seems to be unavoidable.

More precisely,  our approach appeals to the recent 
bounds of Bourgain and Glibichuk~\cite{BG} of multilinear exponential 
sums in arbitrary finite fields which we couple with the classical 
van der Corput differencing. We use this combination to estimate 
exponential sums with polynomials of degree $d$ along affine spaces.

We remark that under some natural conditions the dimension $s$ 
of these spaces can be as low as $r/d$ by the order of magnitude. 
This corresponds exactly to the lowest possible length of 
intervals over $\F_p$ for which one can estimate nontrivially  the 
corresponding exponential sums via Vinogradov's method, see
the recent striking results of Wooley~\cite{Wool1,Wool2,Wool3}.

As in the previous works in this direction, we also give some applications 
of our results.  

Namely, we study the intersection of orbits of 
polynomial dynamical systems and affine spaces and improve and complement 
some results of Roche-Newton and Shparlinski~\cite{RNS}; both are analogues of those of~\cite{Chang1,Chang2,CGOS,GuShp}. We also recall that this question has been introduced
by Silverman and Viray~\cite{SiVi} in characteristic zero 
and then  studied using a very different technique.

Finally, we also consider  the Waring
problem in subspaces.

 We now outline in more detail our main results, that are given in Theorems~\ref{thm:charsumpoly},
 \ref{thm:arbiter} and~\ref{thm:expN} below.

\subsection*{Exponential sums over affine subspaces and polynomial values in affine subspaces} 
Our first motivation is to estimate the number of elements $u$ in an affine subspace $\cA$ of $\L$, that is, a translate of a linear subspace of $\L$, such that $f(u)$ falls also in an affine subspace $\cB$ of $\L$. We denote this number by $\cI_{f}(\cA,\cB)$, that is,
for a nonlinear polynomial $f\in\L[X]$,
\begin{equation}
\label{eq:f(A)B}
\cI_{f}(\cA,\cB)=\#\{u\in\cA\mid f(u)\in \cB\}.
\end{equation}

The basic tool in obtaining estimates for $\cI_{f}(\cA,\cB)$ is using a recent estimate of Bourgain and Glibichuk~\cite[Theorem 4]{BG} on multilinear exponential sums over subsets of $\L$, see Lemma~\ref{lem:BG} below. To arrive to using this result, we apply the classical 
van der Corput differencing method for our exponential sum to reduce the degree of the polynomial $f$, see also~\cite[Theorem C]{Bou}. However, this method was applied so far only with polynomials of degree less than $p$. 

Let $\psi$ be an additive character of $\L$ and $\chi:\L\to\C$  a function satisfying $\chi(x+y)=\chi(x)\chi(y)$, $x,y\in\L$.
The first main result of this paper is obtaining, under certain conditions, estimates for exponential sums over affine subspaces $\cA$ of $\L$ of the type
$$
\sum_{x\in\cA}\chi(x)\psi(f(x)).
$$
What is new about this result is that it applies to several classes of  polynomials of degree larger than $p$ or a multiple of $p$, see Theorem~\ref{thm:charsumpoly}, in contrast to previous results that apply only to polynomial of degree less than $p$.

To estimate $\cI_{f}(\cA,\cB)$, we first use the classical Weil bound to estimate exponential sums, but the bound we obtain is nontrivial only for $s>r(1/2+\varepsilon)$, for some $\varepsilon>0$, and it also applies only for polynomials of degree less than $p$. 
However, applying Theorem~\ref{thm:charsumpoly} we obtain nontrivial estimates for any $s\ge \varepsilon r$, and moreover for more general polynomials of degree larger or equal to $p$, see Theorem~\ref{thm:fBG}. 

The bound of Theorem~\ref{thm:fBG} improves the very recent estimate~\eqref{eq:RNS} 
obtained in~\cite{RNS} for $s< 2.5\(\frac{5}{4}\)^dr\varepsilon$. Moreover, Theorem~\ref{thm:fBG} generalises the result of~\cite{RNS} as this holds only for polynomials of degree $d<p$.

We also conclude from Theorem~\ref{thm:fBG} that, under certain conditions, $f(\cA)$ is not included in any proper affine subspace $\cB$ of $\L$, see Corollary~\ref{cor:fBG}.

\subsection*{Polynomial orbits in affine subspaces}
Given a polynomial $f\in \L[X]$ and an 
element $u \in \L$, we define the orbit
\begin{equation}
\label{eq:Orbu}
\Orb_{f} (u)= \{f^{(n)}(u)\ : \ n = 0, 1,\ldots\},
\end{equation}
where $f^{(n)}$ is the $n$th iterate of $f$, that is,
$$
f^{(0)}=X,\quad f^{(n)}=f(f^{(n-1)}),\quad n\ge 1.
$$ 
As the orbit~\eqref{eq:Orbu} is a subset of $\L$, and thus a finite set,
we denote by $T_{f,u} = \# \Orb_{f} (u)$ 
to be the size of the orbit.

Here we study the frequency of orbit elements that fall in 
an affine subspace of $\L$ considered as a linear vector space over $\K$. 
This question is motivated by a recent work of 
Silverman and Viray~\cite{SiVi} (in characteristic zero 
and using a very different technique).  Recently, several results  have been obtained in~\cite{RNS} using additive combinatorics. Here we improve on several results of~\cite{RNS} and we also extend the class of polynomials to which these results apply, see Corollary~\ref{cor:consorb}.

We also note that the argument of the proof of~\cite[Theorem 6]{RNS} can give information about the frequency of  (not necessarily consecutive) iterates falling in a subspace. We present such a result in Theorem~\ref{thm:arbiter}, as well as apply it to obtain information about intersection of orbits of linearised polynomials in Corollary~\ref{cor:arbiterlin}.

\subsection*{Exponential sums over consecutive integers and the Waring problem}
For a positive integer $n\le p^r-1$, we consider the $p$-adic representation
$$
n=n_0+n_1p+\ldots+n_{s-1}p^{s-1}
$$
for some $s\le r$. 
Let $1\le N\le p^r-1$ and let $f\in\L[X]$ be a polynomial of degree $d$.
Furthermore, let 
$\psi$ be an additive character of $\L$ and let  $\chi:\N\to\C$ 
a $p$-multiplicative function, see Section~\ref{sec:exp consec int} 
for a definition. 
Another main result of this paper is to estimate, using  Theorem~\ref{thm:charsumpoly}, under certain conditions, 
the twisted exponential sum 
$$
S(N)=\sum_{n\le N}\chi(n) \psi(f(\xi_n)),
$$
where $\omega_0,\ldots,\omega_{r-1}$ is a basis of $\L$ over $\F_p$ and 
$$
\xi_n=\sum_{i=0}^{s-1}n_i\omega_i,
$$
which we hope to be of independent interest. 
We present such a result in Theorem~\ref{thm:expN} using the class of $p$-multiplicative functions
$$
\chi(n) = \exp\(2 \pi i \sum_{j=0}^{s-1}\alpha_j n_j\),
$$
where $\alpha_j$, $j=0,1, \ldots$, is a fixed infinite sequence
of real numbers.

Let $f\in\L[X]$ be a polynomial of degree $d$. As another direct consequence of Theorems~\ref{thm:charsumpoly} we prove the existence of a positive integer $k$ such that for any $y\in\L$, the equation 
$$
f(\xi_{n_1})+\ldots+f(\xi_{n_k})=y
$$
is solvable in positive integers $n_1,\ldots,n_k\le N$. We do this first for the case $N=q^s-1$ in Theorem~\ref{thm:warsubsp}, and conclude then Corollary~\ref{cor:War N} for the case $q^{s-1}\le N<q^s$. Recently, 
quite substantial progress has been achieved in the classical Waring problem in finite fields, see~\cite{Cip,CCP,CoPi,vsWoWi}.

We conclude the paper with some remarks and possible extensions of our results, as well as some connections to constructing affine dispersers.

\section{Consecutive differences of polynomials}

For our main results we need a few auxiliary results regarding consecutive differences of polynomials. For a polynomial $f\in\L[X]$ of degree $1\le d<p$ with leading coefficient $a_d$, we define
$$
\Delta_{X_1,X_2}(f)=f(X_1+X_2)-f(X_2)=X_1f(X_1,X_2),
$$
for some polynomial $f(X_1,X_2)\in\L[X_1,X_2]$ of degree 
$$\deg_{X_2}f(X_1,X_2)=d-1
$$ 
with leading coefficient $da_d$. Inductively, we define
\begin{equation}
\label{eq:Delta f}
\begin{split}
\Delta_{X_1,\ldots,X_{k}}&(f)\\
&=f(X_1,\ldots,X_{k-2},X_{k-1}+X_{k})-f(X_1,\ldots,X_{k-2},X_{k})\\
&=X_{k-1}f(X_1,\ldots,X_{k}),
\end{split}
\end{equation}
for some polynomial $f(X_1,\ldots,X_{k})\in\L[X_1,\ldots,X_{k}]$ of degree 
$$\deg_{X_{k}}f=d-(k-1)
$$ 
and leading coefficient with respect to $X_{k}$, $d(d-1)\ldots (d-(k-2))a_d$. We also have the following relation
\begin{equation}
\label{eq:delta k-1 k}
\Delta_{X_1,\ldots,X_{k}}(f)=\frac{1}{X_{k-2}}\(\Delta_{X_1,\ldots,X_{k-2},X_{k-1}+X_k}(f)-\Delta_{X_1,\ldots,X_{k-2},X_{k}}(f)\).
\end{equation}

We now give more details in the following straightforward statement which is 
well-known but is not readily available in the literature.

\begin{lem}
\label{lem:difdeg<p}
Let $f\in\L[X]$ be a polynomial of degree $d<p$ and leading coefficient $a_d\in\L^*$. Then
$$
\Delta_{X_1,\ldots,X_{k}}(f)=X_{k-1}f(X_1,\ldots,X_{k})
$$
where
\begin{equation}
\begin{split}
\label{eq:fk}
f(X_1&,\ldots,X_{k})\\
&=d(d-1)\ldots (d-k+2)a_dX_k^{d-k+1}+\widetilde{f}(X_1,\ldots,X_{k}),
\end{split} 
\end{equation}
for some polynomial $\widetilde{f}(X_1,\ldots,X_{k})\in\L[X_1,\ldots,X_{k}]$ of degrees 
$$\deg_{X_{i}}\widetilde{f}\le d-k+1, \ i=1,\ldots,k-1,\ \deg_{X_{k}}\widetilde{f}\le  d-k.$$ 
\end{lem}

\begin{proof}
The result follows by induction over $k$. Easy computations prove the statement for $k=2$. We assume
it is true for $k-1$ and we prove the statement for $k$. Using the induction hypothesis, we have
\begin{equation}
\label{eq:Dk}
\begin{split}
\Delta&_{X_1,\ldots,X_{k}}(f)\\
&\quad =f(X_1,\ldots,X_{k-2},X_{k-1}+X_k)-f(X_1,\ldots,X_{k-2},X_k)\\
&\quad =d(d-1)\ldots (d-k+3)a_d(X_{k-1}+X_k)^{d-k+2}\\
&\qquad\qquad+\widetilde{f}(X_1,\ldots,X_{k-2},X_{k-1}+X_k)\\
&\qquad\qquad\qquad\quad-d(d-1)\ldots (d-k+3)a_dX_k^{d-k+2}\\
&\qquad\qquad\qquad\qquad\qquad\qquad\qquad-\widetilde{f}(X_1,\ldots,X_{k-2},X_k),
\end{split}
\end{equation}
where $\widetilde{f}(X_1,\ldots,X_{k-2},Y)\in\L[X_1,\ldots,X_{k-2},Y]$ is a polynomial of degrees $$\deg_{X_i}\widetilde{f}(X_1,\ldots,X_{k-2},Y)\le d-k+2,\quad i=1,\ldots,k-2,$$ and $$\deg_{Y}\widetilde{f}(X_1,\ldots,X_{k-2},Y)\le d-k+1.$$  
We  write
\begin{equation}
\label{eq:hk}
\begin{split}
\widetilde{f}(X_1,\ldots,X_{k-2},X_{k-1}+X_k)&=X_{k-1}h(X_1,\ldots,X_k)\\
&\qquad\qquad\quad+\widetilde{f}(X_1,\ldots,X_{k-2},X_k)
\end{split}
\end{equation}
where $h(X_1,\ldots,X_{k})\in\L[X_1,\ldots,X_{k}]$, and taking into account the degrees above, we get
$$\deg_{X_i}h(X_1,\ldots,X_k)\le d-k+1,\quad i=1,\ldots,k-2,$$ and $$\deg_{X_{i}}h(X_1,\ldots,X_k)\le d-k,\quad i=k-1,k.$$
Taking into account~\eqref{eq:Dk} and~\eqref{eq:hk}, and using the binomial expansion of $(X_{k-1}+X_k)^{d-k+2}$, we get
\begin{equation*}
\begin{split}
\Delta_{X_1,\ldots,X_{k}}(f)=X_{k-1}(d(d-1)\ldots (d-k+3)&(d-k+2)a_dX_k^{d-k+1}\\
&\qquad+\widetilde{f}(X_1,\ldots,X_k)),
\end{split}
\end{equation*}
where $\widetilde{f}(X_1,\ldots,X_k)\in\L[X_1,\ldots,X_k]$ is defined by
\begin{equation*}
\begin{split}
\widetilde{f}(X_1,\ldots&,X_k)=h(X_1,\ldots,X_k)\\
&+\frac{(X_{k-1}+X_k)^{d-k+2}-X_k^{d-k+2}-(d-k+2)X_{k-1}X_k^{d-k+1}}{X_{k-1}},
\end{split}
\end{equation*}
and thus satisfy the conditions
$$\deg_{X_{i}}\widetilde{f}\le d-k+1, \quad i=1,\ldots,k-1,\quad \deg_{X_{k}}\widetilde{f}\le d-k.$$ 
We thus  conclude the inductive step.
\end{proof}

If we take $k=d$ in Lemma~\ref{lem:difdeg<p}, 
and then two more consecutive differences, that is, $k=d+1$ and $k=d+2$, we obtain the following consequence.

\begin{cor}
\label{cor:difdeg<p}
We have
$$
\Delta_{X_1,\ldots,X_d}(f)=X_{d-1}\(d!a_dX_d+\widetilde{f}(X_1,\ldots,X_{d-1})\),
$$
where $\widetilde{f}(X_1,\ldots,X_{d-1})\in\L[X_1,\ldots,X_{d-1}]$ is of degrees $$\deg_{X_i} \widetilde{f}(X_1,\ldots,X_{d-1})\le 1, \qquad i=1,\ldots,d-1,
$$ 
and thus
$$
\Delta_{X_1,\ldots,X_d,X_{d+1}}(f)=d!a_dX_d, \ \Delta_{X_1,\ldots,X_{d+1},X_{d+2}}(f)=0.
$$
\end{cor}

\begin{lem}
\label{lem:xpg}
Let $\nu\ge 1$ and $f=X^{p^{\nu}}g\in\L[X]$, where $g\in\L[X]$ is of degree $d<p$. Then
$$
\Delta_{X_1,\ldots,X_{d+3}}(f)=0.
$$
\end{lem}
\begin{proof}
We prove by induction over $k\ge 2$ that
\begin{equation}
\label{eq:fik}
\begin{split}
\Delta&_{X_1,\ldots,X_{k}}(f)=X_{k-1}\(\sum_{j=1}^{k}X_j^{p^{\nu}}\)g(X_1,\ldots,X_{k})\\
&\qquad+X_{k-1}\sum_{j=1}^{k-1} X_j^{p^{\nu}-1}g(X_1,\ldots,X_{j-1},X_{j+1},\ldots,X_{k-1},X_{k}),
\end{split}
\end{equation}
where $g(X_1,\ldots,X_k)$ and $g(X_1,\ldots,X_{j-1},X_{j+1},\ldots,X_{k-1},X_{k})$ are defined by~\eqref{eq:Delta f} and~\eqref{eq:fk}.

For $k=2$, the computations follow exactly as for the general case, so to avoid repetition we prove only the induction step from $k-1$ to $k$. Using~\eqref{eq:Delta f},~\eqref{eq:delta k-1 k} and the induction step, we have
\begin{equation*}
\begin{split}
&\Delta_{X_1,\ldots,X_{k}}(f)=f(X_1,\ldots,X_{k-2},X_{k-1}+X_{k})-f(X_1,\ldots,X_{k-2},X_{k})\\
&\qquad\qquad=\frac{1}{X_{k-2}}\(\Delta_{X_1,\ldots,X_{k-2},X_{k-1}+X_{k}}(f)-\Delta_{X_1,\ldots,X_{k-2},X_{k}}(f)\)\\
&\qquad\qquad=\(\sum_{j=1}^{k-1}X_j^{p^{\nu}}+X_{k}^{p^{\nu}}\)g(X_1,\ldots,X_{k-2},X_{k-1}+X_{k})\\
&\qquad\qquad\qquad+\sum_{j=1}^{k-2} X_j^{p^{\nu}-1}g(X_1,\ldots,X_{j-1},X_{j+1},\ldots,X_{k-2},X_{k-1}+X_{k})\\
&\qquad\qquad\qquad\qquad-\(\sum_{j=1}^{k-2}X_j^{p^{\nu}}+X_{k}^{p^{\nu}}\)g(X_1,\ldots,X_{k-2},X_{k})\\
&\qquad\qquad\qquad\qquad-\sum_{j=1}^{k-2} X_j^{p^{\nu}-1}g(X_1,\ldots,X_{j-1},X_{j+1},\ldots,X_{k-2},X_{k}).
\end{split}
\end{equation*}
Writing now, as in Lemma~\ref{lem:difdeg<p} (applied to $g$ in place of $f$),
\begin{equation*}
\begin{split}
g(X_1,\ldots,&X_{j-1},X_{j+1},\ldots,X_{k-2},X_{k-1}+X_{k})\\
&\qquad=X_{k-1}g(X_1,\ldots,X_{j-1},X_{j+1},\ldots,X_{k})\\
&\qquad\qquad\qquad\qquad+g(X_1,\ldots,X_{j-1},X_{j+1},\ldots,X_{k-2},X_{k}),
\end{split}
\end{equation*}
and the same for $g(X_1,\ldots,X_{k-2},X_{k-1}+X_{k})$, and making simple computations, 
we get the equation~\eqref{eq:fik}. In fact, using Lemma~\ref{lem:difdeg<p}, 
and the fact that, by~\eqref{eq:Delta f} we have
$$
\Delta_{X_1,\ldots,X_{k}}(g)=X_{k-1}g(X_1,\ldots,X_{k})
$$
and
\begin{equation*}
\begin{split}
\Delta_{X_1,\ldots,X_{j-1},X_{j+1},\ldots,X_{k-1},X_{k}}&(g)\\
&=X_{k-1}g(X_1,\ldots,X_{j-1},X_{j+1},\ldots,X_{k-1},X_{k}),
\end{split}
\end{equation*}
one can rewrite the equation~\eqref{eq:fik} as follows
\begin{equation*}
\begin{split}
\Delta_{X_1,\ldots,X_{k}}&(f)=\Delta_{X_1,\ldots,X_k}(g)\(\sum_{j=1}^{k}X_j^{p^{\nu}}\)\\
&\qquad\qquad+\sum_{j=1}^{k-2} X_j^{p^{\nu}-1}\Delta_{X_1,\ldots,X_{j-1},X_{j+1},\ldots,X_{k-1},X_{k}}(g)\\
&\qquad\qquad\qquad\qquad\qquad+X_{k-1}^{p^{\nu}}g(X_1,\ldots,X_{k-2},X_{k}).
\end{split}
\end{equation*}

By Corollary~\ref{cor:difdeg<p} we have
$$
\Delta_{X_1,\ldots,X_{d+1},X_{d+2}}(g)=0,  \qquad \Delta_{X_1,\ldots,X_{j-1},X_{j+1},\ldots,X_{d+2},X_{d+3}}(g)=0, 
$$
and similarly $\Delta_{X_1,\ldots,X_{d+1},X_{d+3}}(g)=0$, and thus,
$$
g(X_1,\ldots,X_{d+1},X_{d+3})=0,
$$
and thus we conclude the proof.
\end{proof}

As a direct application of Lemma~\ref{lem:xpg}, we obtain the following more general result.
\begin{cor}
\label{cor:xpf}
Let $$f=X^{p^{\nu}}g_{\nu}+\ldots+X^pg_1+g_0,\ g_i\in\L[X],\ i=0\ldots,\nu,
$$
with
$$
\deg g_i+3\le \deg g_0=d<p,\ i=1\ldots,\nu,$$
and with $g_0$ having leading coefficient $a_d$. Then, 
$$
\Delta_{X_1,\ldots,X_d}(f)=X_{d-1}\(d!a_dX_d+\widetilde{f}(X_1,\ldots,X_{d-1})\),
$$
where $\widetilde{f}(X_1,\ldots,X_{d-1})\in\L[X_1,\ldots,X_{d-1}]$ is of degrees $$\deg_{X_i} \widetilde{f}(X_1,\ldots,X_{d-1})\le 1,$$ $i=1,\ldots,d-1$.
\end{cor}
\begin{proof}
The proof follows directly from Lemma~\ref{lem:xpg}. Indeed, we have
\begin{equation*}
\begin{split}
\Delta_{X_1,\ldots,X_{k}}(f)&=\sum_{i=1}^{\nu}\Delta_{X_1,\ldots,X_{k}}\(X^{p^i}g_i\)+\Delta_{X_1,\ldots,X_{k}}(g_0).
\end{split}
\end{equation*}
As $d\ge \deg g_i+3$, from Lemma~\ref{lem:xpg} 
we obtain
$$
\Delta_{X_1,\ldots,X_{d}}\(X^{p^i}g_i\)=0,
$$
and thus, 
$$
\Delta_{X_1,\ldots,X_{d}}(f)=\Delta_{X_1,\ldots,X_{d}}(g_0).
$$
Applying now Corollary~\ref{cor:difdeg<p}, we conclude the proof.
\end{proof}

\begin{remark}
We note that when $g_i$, $i=1,\ldots,\nu$, are constant polynomials in Corollary~\ref{cor:xpf}, we are in the case
$$f=c_{\nu}X^{p^{\nu}}+\ldots+c_1X^p+g_0,\ g_i\in\L[X],\ c_i\in\L, \ i=1\ldots,\nu,
$$
with
$$
3\le \deg g_0=d<p.$$
Then, we need to take only two differences to eliminate the power of $p$ monomials, that is, we have
$$
\Delta_{X_1,X_2,X_3}=X_2g_0(X_1,X_2,X_3),
$$
where $g_0(X_1,X_2,X_3)\in\L[X_1,X_2,X_3]$ with $\deg_{X_3}g_0(X_1,X_2,X_3)=d-2$.
\end{remark}

\begin{lem}
\label{lem:p g}
Let 
$$f=X^{p^{\nu}+p^{\nu-1}+\ldots+p+1}+g\in\L[X],$$ where $g\in\L[X]$ is a polynomial of degree $4\le d<p$ with leading coefficient $a_d$. Then
$$
\Delta_{X_1,\ldots,X_{d}}(f)=X_{d-1}\(d!a_dX_d+\widetilde{f}(X_1,\ldots,X_{d-1})\),
$$
where $\widetilde{f}(X_1,\ldots,X_{d-1})\in\L[X_1,\ldots,X_{d-1}]$ is of degrees $$\deg_{X_i} \widetilde{f}(X_1,\ldots,X_{d-1})\le 1,$$ $i=1,\ldots,d-1$.
\end{lem}

\begin{proof}
The case $\nu=1$ is a special case of  Corollary~\ref{cor:xpf}. The proof for $\nu>1$ follows also the same as the proof of Lemma~\ref{lem:xpg}, only that one needs to take three differences to eliminate all powers of $p$ in the degree and get a polynomial of degree less then $p$. Indeed, we denote the monomial
$$
m=X^{p^{\nu}+p^{\nu-1}+\ldots+p+1}.
$$

Simple computations show that, as in Lemma~\ref{lem:xpg}, we have
$$
\Delta_{X_1,X_2}(m)=X_1m(X_1,X_2),
$$
where
$$
m(X_1,X_2)=X_1^{p^{\nu+\ldots+p}}+X_2^{p^{\nu+\ldots+p}}+X_2X_1^{p^{\nu+\ldots+p}-1}.
$$

Next, we have
$$
\Delta_{X_1,X_2,X_3}(m)=X_2m(X_1,X_2,X_3),
$$
where 
$$
m(X_1,X_2,X_3)=X_1^{p^{\nu+\ldots+p}-1}+X_2^{p^{\nu+\ldots+p}-1}.
$$
At the next step we already get
$$
\Delta_{X_1,X_2,X_3,X_4}(m)=0.
$$
Thus, as $d\ge 4, $we have
$$
\Delta_{X_1,X_2,X_3,X_4}(f)=\Delta_{X_1,X_2,X_3,X_4}(g), \ \deg_{X_4} \Delta_{X_1,X_2,X_3,X_4}(g)=d-3\ge 1,
$$
and the result follows by applying Lemma~\ref{lem:difdeg<p}.
\end{proof}

\begin{lem}
\label{lem:g(l)}
Let $f=g(l(x))\in\L[X]$,  
where $g\in\L[X]$ is a polynomial of degree $d<p$ with leading coefficient $a_d$, and
\begin{equation}
\label{eq:lin}
 l=\sum_{i=1}^{\nu}b_iX^{p^i}\in\L[X],\ p^{\nu}<q^r,
 \end{equation}
is a $p$-polynomial polynomial. Then
$$
\Delta_{X_1,\ldots,X_{d}}(f)=l(X_{d-1})\(d!a_dl(X_d)+\widetilde{f}\(l(X_1),\ldots,l(X_{d-1})\)\),
$$
where $\widetilde{f}(X_1,\ldots,X_{d-1})\in\L[X_1,\ldots,X_{d-1}]$ is of degrees $$\deg_{X_i} \widetilde{f}(X_1,\ldots,X_{d-1})\le 1,$$ $i=1,\ldots,d-1$.
\end{lem}
\begin{proof}
As the polynomial $l$ is additive, that is $l(X_1+X_2)=l(X_1)+l(X_2)$, then the proof follows exactly as the proof of
Lemma~\ref{lem:xpg} and Lemma~\ref{lem:p g}.
\end{proof}

\section{Exponential sums over subspaces}

First we introduce the following:

\begin{definition}
\label{eq:etagood}
For $0<\eta\le 1$, we define a subset $A\subset\L$ to be {\it $\eta$-good} if 
$$
\#\(A\cap b\F\)\le \(\#A\)^{1-\eta}
$$
for any element $b$ and a proper subfield $\F$ of $\L$.
\end{definition}

The main result of this section follows from the following 
estimate due to Bourgain and Glibichuk~\cite[Theorem 4]{BG} which applies to $\eta$-good sets.

For $0<\eta\le 1$, we define 
\begin{equation}
\label{eq:gamma}
\gamma_{\eta}=\min\(\frac{1}{156450},\frac{\eta}{120}\).
\end{equation}

Also, all over the paper  $\psi$ represents an additive character of $\L$.

\begin{lem}
\label{lem:BG}
Let $3\le n\le 0.9\log_2 (r\log_2 q)$ and $A_1,A_2,\ldots,A_n\subseteq \L^*$. Let $0<\eta\le 1$ and $\gamma_{\eta}$ be defined by~\eqref{eq:gamma}. Suppose $\# A_i\ge 3$, $i=1,2,\ldots,n$, and that for every $j=3,4,\ldots,n$
the sets $A_j$ are $\eta$-good. 
Assume further that
$$
\# A_1 \# A_2\(\#A_3 \cdots \# A_n\)^{\gamma_{\eta}}>q^{r(1+\varepsilon)}
$$
for some $\varepsilon>0$. Then, for sufficiently large $q$, we have the estimate
$$
\left|\sum_{a_1\in A_1}\sum_{a_2\in A_2}\cdots \sum_{a_n\in A_n}\psi(a_1a_2\ldots a_n)\right|< 100\#A_1\#A_2\cdots \#A_nq^{-0.45 r\varepsilon/2^n}.
$$
\end{lem}

For our results we need an estimate for slightly different (and possibly larger) sums
$$
\sum_{a_2\in A_2}\cdots \sum_{a_n\in A_n}\left|\sum_{a_1\in A_1}\psi(a_1a_2\ldots a_n)\right|,
$$
which we derive directly from Lemma~\ref{lem:BG}. 
We record this estimate for more general  weighted sums, which may 
be of independent interest for some future applications.

\begin{cor}
\label{cor:BGw}
Let $4\le n\le 0.9\log_2 (r\log_2 q)+1$ and $A_1,A_2,\ldots,A_n\subseteq \L^*$.
Let $0<\eta\le 1$ and $\gamma_{\eta}$ be defined by~\eqref{eq:gamma}. Suppose $\# A_i\ge 3$, $i=2,\ldots,n$, and that for every $j=4,\ldots,n$ the sets $A_j$ are $\eta$-good. 
Assume further that
\begin{equation}
\label{eq:setscond}
\# A_2 \# A_3\(\#A_4 \cdots \# A_n\)^{\gamma_{\eta}}>q^{r(1+\varepsilon)}
\end{equation}
for some $\varepsilon>0$. 
Let the weights $w_i:\L\to \C$, $i=1,\ldots,n$,
be such that
\begin{equation}
\label{eq:wcond}
\sum_{a_i\in A_i}|w_i(a_i)|^2\le B_i,\quad  i=1,\ldots,n.
\end{equation}
Then, for sufficiently large $q$, for the sum
$$
\cJ=\sum_{a_2\in A_2}\cdots \sum_{a_{n}\in A_{n}}w_2(a_2)\ldots w_n(a_n)\left|\sum_{a_1\in A_1}w_1(a_1)\psi(a_1a_2\ldots a_n)\right|,
$$
we have the estimate
$$
|\cJ|<\prod_{i=1}^n(B_i\#A_i)^{1/2}\((\#A_1)^{-1/2}+10q^{-0.45 r\varepsilon/2^n}\).
$$
\end{cor}
\begin{proof}
Squaring and applying the Cauchy-Schwarz inequality, we get
\begin{equation*}
\begin{split}
|\cJ|^2&\le \sum_{a_2\in A_2}\cdots \sum_{a_n\in A_n}|w_2(a_2)|^2\ldots |w_n(a_n)|^2\\
&\qquad\qquad\qquad\qquad\quad\cdot\sum_{a_2\in A_2}\cdots \sum_{a_n\in A_n}\left|\sum_{a_1\in A_1}w_1(a_1)\psi(a_1a_2\ldots a_n)\right|^2\\
&= \prod_{i=2}^n\(\sum_{a_i\in A_i}|w_i(a_i)|^2\)\sum_{a_1,b_1\in A_1}w_1(a_1)\overline{w_1(b_1)}\\
&\qquad\qquad\qquad\qquad\qquad\qquad\cdot\sum_{a_2\in A_2}\cdots \sum_{a_n\in A_n}\psi(a_2\ldots a_n(a_1-b_1)).
\end{split}
\end{equation*}
Applying~\eqref{eq:wcond} and Lemma~\ref{lem:BG} (for the sum over $A_2,\ldots,A_n$), we obtain
\begin{equation*}
\begin{split}
|\cJ|&\le \prod_{i=2}^nB_i^{1/2}\left(B_1^{1/2}\prod_{i=2}^n(\#A_i)^{1/2}\right.\\
&\qquad\qquad\left.+10\prod_{i=2}^n\(\#A_i\)^{1/2}q^{-0.45 r\varepsilon/2^n}\(\sum_{a_1,b_1\in A_1}w_1(a_1)\overline{w_1(b_1)}\)^{1/2}\right),
\end{split}
\end{equation*}
where the first summand comes from the case $a_1=b_1$.
Taking into account that
$$
\sum_{a_1,b_1\in A_1}w_1(a_1)\overline{w_1(b_1)}= \left|\sum_{a\in A_1}w_1(a)\right|^2\le \#A_1\sum_{a\in A_1}|w_1(a)|^2,
$$
and using again~\eqref{eq:wcond}, we conclude the proof.
\end{proof}

Throughout the paper, we slightly abuse this notion of $\eta$-good sets and 
in the case of affine spaces introduce the following:

\begin{definition}
\label{eq:etagood aff}
We say that an affine subspace $\cA\subset\L$ is 
 $\eta$-good if $\cA = \cL+a$ for some $a\in \L$ and  linear
 subspace $\cL \subseteq \L$ which is $\eta$-good 
 as in Definition~\ref{eq:etagood}. 
 \end{definition}

Using 
Lemma~\ref{lem:BG}
we prove the following estimate of additive 
character sums with polynomial argument over an affine subspace of $\L$, which can be seen as an explicit version of~\cite[Theorem C]{Bou}. However, we notice that all previous such estimates are known for polynomials of degree less than $p$. Here we obtain results for more general polynomials.

For $0<\varepsilon,\eta\le 1$, we define
\begin{equation}
\label{eq:delta}
\delta(\varepsilon,\eta)=\max\(4,\gamma_{\eta}^{-1}(\varepsilon^{-1}-1)+3\).
\end{equation}

\begin{thm}
\label{thm:charsumpoly}
Let $0<\varepsilon,\eta\le 1$ be arbitrary numbers, $\gamma_{\eta}$
and 
$\delta(\varepsilon,\eta)$ be defined by~\eqref{eq:gamma} and~\eqref{eq:delta}, respectively. 
Let $\cA\subseteq\L$ be an $\eta$-good affine subspace of dimension $s$ over $\K$ with
$$
s\ge \varepsilon r.
$$
Let $d$ be an integer satisfying the inequalities 
\begin{equation}
\label{eq:h1}
\delta(\varepsilon,\eta)\le d\le  \min\(p, 0.9\log_2\log_2q^r\)+1,
\end{equation}
and let
$f$ be any polynomial of one of the following forms: 
\begin{itemize}
\item[(i)] $f=X^{p^{\nu}}g_{\nu}+\ldots+X^pg_1+g_0$, where  $g_i\in\L[X]$ , $i=0\ldots,\nu,
$ are such that
$$
\deg g_0=d\ge \deg g_i+3,\qquad  i=1\ldots,\nu;
$$  

\item[(ii)] $f=X^{p^{\nu}+p^{\nu-1}+\ldots+p+1}+g\in\L[X]$, where $g\in\L[X]$ with $\deg g = d\ge 4$;
\item[(iii)] $f=g(l(X))\in\L[X]$,  
where $g\in\L[X]$ with $\deg g = d$ and $l\in\L[X]$ is a permutation $p$-polynomial of the form~\eqref{eq:lin} such that $l(\cL)$ is $\eta$-good.
\end{itemize}
Let $\chi:\L\to \C$ a function satisfying $\chi(x+y)=\chi(x)\chi(y)$, $x,y\in\L$, and such that
$$
\sum_{x\in\cA}|\chi(x)|^{2^{d}}\le B.
$$
Then, for sufficiently large $q$,
we have
$$
\left|\sum_{x\in\cA}\chi(x)\psi(f(x))\right|\le 2B^{(d+1)/2^{d+1}}q^{s\(1-(d+1)/2^{d+1}\)-r\vartheta},
$$
where 
\begin{equation}
\label{eq:theta}
\vartheta=\frac{0.9\varepsilon}{2^{2d}}.
\end{equation}
\end{thm}

\begin{proof} 
Write $\cA=a+\cL$, $a\in\L$, where $\cL$ is a $\K$-linear space of $\dim_{\K}\cL=s$. Making the linear transformation $x\in\cL \to a+x$, we reduce the problem to estimating the character sum over a linear subspace,
$$
S=\sum_{x\in\cL}\chi(x)\psi\(f(x)\).
$$

We use the method  in~\cite{Weyl}. For this we square the sum and after changing the order of summation and substituting $x_1\to x_1+x_2$, we get
\begin{equation*}
\begin{split}
|S|^2&=\sum_{x_2\in\cL}\sum_{x_1\in\cL}\chi(x_1)\overline{\chi(x_2)}\psi\(f(x_1)-f(x_2)\)\\
&\qquad\qquad\qquad\qquad\qquad=\sum_{x_1\in\cL}\chi(x_1)\sum_{x_2\in\cL}|\chi(x_2)|^2\psi\(\Delta_{x_1,x_2}(f)\),
\end{split}
\end{equation*}
where $\Delta_{X_1,X_2}(f)$ is defined by~\eqref{eq:Delta f}. 

Squaring and applying the Cauchy-Schwarz inequality again, we get
\begin{equation*}
\begin{split}
|S|^4&\le q^s\sum_{x_1\in\cL}|\chi(x_1)|^2\left|\sum_{x_2\in\cL}|\chi(x_2)|^2\psi\(\Delta_{x_1,x_2}(f)\)\right|^2\\
&=q^s\sum_{x_1\in\cL}|\chi(x_1)|^2\sum_{x_2,x_3\in\cL}|\chi(x_2)|^2|\chi(x_3)|^2\psi\(\Delta_{x_1,x_2}(f)-\Delta_{x_1,x_3}(f)\).
\end{split}
\end{equation*}
Substituting $x_2\to x_2+x_3$, we get
\begin{equation*}
\begin{split}
|S|^4&\le q^s\sum_{x_1,x_2\in\cL}|\chi(x_1)|^2|\chi(x_2)|^2\sum_{x_3\in\cL}|\chi(x_3)|^4\psi\(x_1\Delta_{x_1,x_2,x_3}(f)\),
\end{split}
\end{equation*}
where $\Delta_{X_1,X_2,X_3}(f)$ is defined by~\eqref{eq:Delta f}.

Simple inductive argument shows that 
applying this procedure, that is, squaring and applying the Cauchy-Schwarz inequality, $d-1$ times, and using 
Corollary~\ref{cor:xpf} and Lemma \ref{lem:p g} for the polynomial $f$ corresponding to (i) and (ii), we get
to the exponential sum
\begin{equation*}
\begin{split}
|S|^{2^{d-1}}&\le q^{s(2^{d-1}-d)}\sum_{x_1,x_2,\ldots,x_{d-1}\in\cL}|\chi(x_1)|^{2^{d-2}}|\chi(x_2)|^{2^{d-2}}\ldots|\chi(x_{d-1})|^{2^{d-2}}\\
&\qquad\qquad\qquad\qquad\qquad\quad\cdot\left|\sum_{x_d\in\cL}|\chi(x_d)|^{2^{d-1}}\psi\(d!a_dx_1x_2\ldots x_{d}\)\right|.
\end{split}
\end{equation*}

We apply now Corollary~\ref{cor:BGw} with 
$$4\le n=d\le 0.9\log_2\log_2q^r+1
$$
and $A_1=\ldots=A_n=\cL$. We note that the condition~\eqref{eq:h1} implies that
$$
s\(2+(d-3)\gamma_{\eta}\)>r(\varepsilon+1),
$$
and thus the condition~\eqref{eq:setscond} in Corollary~\ref{cor:BGw} is also satisfied. Moreover, we have
$$
\(\sum_{x\in\cL}|\chi(x)|^{2^{d-1}}\)^2\le q^s\sum_{x\in\cL}|\chi(x)|^{2^{d}}\le Bq^s,
$$
and thus
$$
\sum_{x\in\cL}|\chi(x)|^{2^{d-1}}\le B^{1/2}q^{s/2}.
$$
We get
$$
|S|^{2^{d-1}} \le 10 B^{(d+1)/4}q^{s\(2^{d-1}-(d+1)/4\)- 0.45 r\varepsilon/2^d}.
$$
which immediately implies the result. 

For the case (iii), proceeding the same but applying Lemma~\ref{lem:g(l)} and taking into account that $l$ is a permutation polynomial, we get to the exponential sum
\begin{equation*}
\begin{split}
|S|^{2^{d-1}}&\le q^{s(2^{d-1}-d)}\sum_{x_1,x_2,\ldots,x_{d-1}\in\cL}|\chi(x_1)|^{2^{d-2}}\chi(x_2)|^{2^{d-2}}\ldots|\chi(x_{d-1})|^{2^{d-2}}\\
&\qquad\qquad\qquad\qquad\cdot\left|\sum_{x_d\in\cL}|\chi(x_d)|^{2^{d-1}}\psi\(d!a_dl(x_1)l(x_2)\ldots l(x_{d})\)\right|\\
&=q^{s(2^{d-1}-d)}\sum_{x_1,x_2,\ldots,x_{d-1}\in l(\cL)}|\chi(l^{-1}(x_1))|^{2^{d-2}}\chi(l^{-1}(x_2))|^{2^{d-2}}\ldots\\
&\quad\cdot|\chi(l^{-1}(x_{d-1}))|^{2^{d-2}}\left|\sum_{x_d\in l(\cL)}|\chi(l^{-1}(x_d))|^{2^{d-1}}\psi\(d!a_dx_1x_2\ldots x_{d}\)\right|
\end{split}
\end{equation*}
where $l(\cL)$ is a subset of $\L$ of cardinality $q^s$ (we note that if $l$ is a $q$-polynomial, then $l(\cL)$ is actually a $\K$-linear subspace of $\L$) and $l^{-1}$ is the compositional inverse of $l$, which is again a linearised polynomial of the form~\eqref{eq:lin}, see~\cite[Theorem 4.8]{WuLi}. Thus, we have that
$$
\chi\(l^{-1}(x+y)\)=\chi\(l^{-1}(x)+l^{-1}(y)\)=\chi\(l^{-1}(x)\)\chi\(l^{-1}(y)\)
$$
and
$$
\sum_{x\in l(\cL)}|\chi\(l^{-1}(x)\)|^{2^{d}}=\sum_{x\in \cL}|\chi\(x\)|^{2^{d}}\le B.
$$
As we also assume that $l(\cL)$ is $\eta$-good, the estimate follows the same by applying Corollary~\ref{cor:BGw}.
\end{proof}

\begin{remark}
 We note that, by~\cite[Theorem 7.9]{LN}, a $p$-polynomial $l\in\L[X]$ as defined by~\eqref{eq:lin} is a permutation polynomial  if and only if it has only the root $0$  in $\L$. 
\end{remark}

\begin{remark}
 In Theorem~\ref{thm:charsumpoly}, (iii), we assume that the set $l(\cL)$ is $\eta$-good. We note that when $l(X)=X^{p^{\nu}}$, then $l(b\F)=b^{p^{\nu}}\F$, for any subfield $\F$ of $\L$ and any element $b$ not in $\F$, and thus $l(\cL)$ is $\eta$-good.
Another immediate example can be given for a prime $q=p$ and also a prime  $r$. 
Since the only proper subfield $\F$ of $\L$ is $\F_p$, if $s\ge 2$, that is, $\#\cL\ge p^2$ then, 
$$
\#\(l(\cL)\cap b\F\)\le p\le \#l(\cL)^{1/2}.
$$
\end{remark}

\begin{remark}
\label{rem:chi} Probably the most natural examples of the function $\chi(x)$ in 
Theorem~\ref{thm:charsumpoly} is given by exponential functions such as 
$$
\chi(x) = \exp\(2 \pi  \sum_{j=1}^m \zeta_j \Tr_{L|\F_p}(\tau_j x) \)
$$
for some $\zeta_j \in \R$ and $\tau_j \in \L$. 
\end{remark}

\section{Values of polynomials in subspaces}

In this section we give upper bounds for $\cI_f(\cA,\cB)$ defined by~\eqref{eq:f(A)B}, that is, the cardinality of $f(\cA)\cap \cB$, for a polynomial $f\in\L[X]$ and affine subspaces $\cA, \cB$ of $\L$ over $\K$.

For our first result we use the Weil bound, see~\cite[Theorem 5.38]{LN}, 
in a standard way. We recall it for the sake of completeness and use it 
as a benchmark for further improvements.

\begin{lem}
\label{lem:Weil}
Let $f\in\L[X]$ be of degree $d\ge 1$ with $(d,p)=1$, and let $\psi$ be a nontrivial additive character of $\L$. Then
$$
\left|\sum_{c\in\L}\psi(f(c))\right|\le (d-1)q^{r/2}.
$$
\end{lem}

\begin{thm}
\label{thm:genf}
Let $f\in\L[X]$ be a polynomial of degree $d\ge 2$, $(d,p)=1$, $\cA\subseteq\L$ and $\cB\subseteq\L$ affine subspaces of dimension $s$ and $m$, respectively, over $\K$. Then, we have
$$
\cI_f(\cA,\cB)=q^{s+m-r}+O\(dq^{r/2}\).
$$
\end{thm}
\begin{proof}
As in Theorem~\ref{thm:charsumpoly}, we can reduce the problem to estimating $\cI_f(\cL_1,\cL_2)$, where $\cL_1,\cL_2$ are linear spaces.

Let $\beta_1,\ldots,\beta_{r-s}$ be the basis for the complementary space of $\cL_1$ and $\omega_1,\ldots,\omega_{r-m}$ be the basis for the complementary space of $\cL_2$, that is, $u\in\cL_1$ and $v\in\cL_2$ if and only if
\begin{equation}
\label{eq:Tr=0}
\Tr_{\L|\K}(\beta_iu)=0, \ i=1,\ldots,r-s,\quad  \Tr_{\L|\K}(\omega_iv)=0, \ i=1,\ldots,r-m.
\end{equation}

We use the relations~\eqref{eq:Tr=0} to give an upper bound for $\cI_f(\cA,\cB)=\cI_f(\cL_1,\cL_2)$.
Indeed, let 
$\psi$ be a nontrivial additive character of $\L$. Using additive character sums to count the elements $u\in\cL_1$ such that the elements of $f(u)$ satisfy~\eqref{eq:Tr=0}, we have
\begin{equation*}
\begin{split}
\cI_f(\cL)&=\frac{1}{q^{2r-s-m}}\sum_{x\in\L}\sum_{\substack{c_{i}, d_j\in\K\\ i=1,\ldots,r-s\\ j=1,\ldots,r-m}}\psi\(\sum_{i=1}^{r-s}c_{i}\beta_i x+\sum_{j=1}^{r-m}d_{j}\omega_jf(x)\)\\
&=q^{s+m-r}+\frac{1}{q^{2r-s-m}}\sum_{\substack{c_{i}, d_i\in\K\\ i=1,\ldots,r-s\\ j=1,\ldots,r-m}}\hskip-22 pt {\phantom{{\Sigma^2}}}^*\sum_{x\in\L}\psi\(\sum_{i=1}^{r-s}c_{i}\beta_i x+\sum_{j=1}^{r-m}d_{j}\omega_jf(x)\),
\end{split}
\end{equation*}
where the first term is given by $c_{i}=d_{j}=0$, $i=1,\ldots,r-s$, $j=1,\ldots,r-m$, and $\sum^{*}$ means that at least one element $c_{i}, d_j\ne0$.

We notice that since $\deg f=d\ge 2$, nontrivial linear combinations
$$\sum_{i=1}^{r-s}c_{i}\beta_i x+\sum_{j=1}^{r-m}d_{j}\omega_jf(x), \ c_{i}, d_j\in\K,\ i=1,\ldots,r-s,\ j=1,\ldots,r-m,
$$ 
that appear in the inner sum 
are all nonconstant polynomials. Indeed, assume that this
is not the case, and without loss of generality we can also assume that $d_i\ne 0$ for at least one $i=1,\ldots,t$. Then, the vanishing of the leading coefficients (of the monomial $X^d$)
$$
\sum_{i=1}^{r-m}d_i\omega_iX^{d}=0
$$
implies that the elements $\omega_1,\ldots,\omega_{r-m}$ are linearly dependent as elements of $\L$ seen as a vector space over $\K$, which contradicts the hypothesis. 

We can apply now the Weil bound given by Lemma~\ref{lem:Weil} to the sum over $x\in\L$ and conclude the proof.
\end{proof}
We note that the bound of Theorem~\ref{thm:genf} is nontrivial whenever $dq^{r/2}<q^s$, and thus only for $s>r(1/2+\varepsilon)$, for some $\varepsilon>0$. 

In the rest of this section we obtain a bound that depends on both parameters $s$ and $m$. We recall first a similar result that was recently obtained in~\cite[Theorem 7]{RNS} for the case $\cA=\cB$ and only for polynomials of degree smaller than $p$.

Let $f\in \L[X]$ be of degree $d=\deg f$ 
with $p> d \ge 2$ and let $\cA \subseteq \L$ be an  affine subspace of dimension $s$ over $\K$
such that for any subfield $\F \subseteq \L$ and any $b\in \L$ we have
$$
\#\(\cL \cap b\F  \) \leq  \max{\left\{(\#\L)^{1/2},\frac{q^{s(1-\rho_d)}}{8}\right\}}, 
$$ 
where $\cA = a + \cL$ for some $a\in \F$ and a linear subspace $\cL \subseteq \L$. 
Then the following estimate is obtained in~\cite[Theorem 7]{RNS}:
\begin{equation}
\label{eq:RNS}
\cI_f(\cA,\cA) \ll  q^{s(1-\kappa_d)}, 
\end{equation}
with 
$$
\eta_d=\frac{4}{277\cdot 5^{d-2}-1},\quad \kappa_d=\frac{4}{277\cdot 5^{d-2}+3},
$$
and for $d\ge 3$,
$$
\vartheta_d = \eta_d+\vartheta_{d-1}  - \eta_d\vartheta_{d-1},\quad \rho_d = \eta_d+\vartheta_d -  \eta_d\vartheta_d, 
$$
where $\eta_2=\vartheta_2 =  1/69$.

We prove now one of our main results using Theorem~\ref{thm:charsumpoly}.

\begin{thm}
\label{thm:fBG}
Let $0<\varepsilon,\eta\le 1$ be arbitrary numbers and let $\gamma_{\eta}$ and 
$\delta(\varepsilon,\eta)$ be defined by~\eqref{eq:gamma} and~\eqref{eq:delta}, respectively.
Let $\cA\subseteq\L$ be an $\eta$-good affine subspace of dimension $s$ over $\K$  with
$$
s\ge \varepsilon r,
$$
and $\cB\subseteq\L$ another affine subspace of dimension $m$ over $\K$. 
Let $d$ 
and 
$f$ be as in Theorem~\ref{thm:charsumpoly}. 
Then 
$$
\left|\cI_f(\cA,\cB) - q^{s+m-r}\right|\le 2q^{s-r\vartheta},
$$
where $\vartheta$ is defined by~\eqref{eq:theta}. 
\end{thm}

\begin{proof}
As in the proof of Theorem~\ref{thm:genf} (where we take the sum over $\cL_1$, not all the field $\L$), we have
\begin{equation*}
\begin{split}
\cI_f(\cA,\cB)&=\frac{1}{q^{r-m}}\sum_{\substack{c_{i}\in\K\\ i=1,\ldots,r-m}}\sum_{x\in\cL_1}\psi\(\sum_{i=1}^{r-m}c_{i}\omega_if(x)\),\\
&=q^{s+m-r}+\frac{1}{q^{r-m}}\sum_{\substack{c_{i}\in\K\\ i=1,\ldots,r-m}}\hskip-20 pt {\phantom{{\Sigma^2}}}^*\sum_{x\in\cL_1}\psi\(\sum_{i=1}^{r-m}c_{i}\omega_if(x)\),  
\end{split}
\end{equation*}
where 
the first term corresponds to $c_{i}=0$ for all $i=1,\ldots, r-m$ and $\sum^{*}$ means that at least one $c_i\ne 0$. We denote
$$
T=\sum_{x\in\cL_1}\psi\(\sum_{i=1}^{r-m}c_{i}\omega_if(x)\).
$$

We apply Theorem~\ref{thm:charsumpoly} (with $\chi(x)=1$, $x\in\L$, and $B=q^s$) for the sum $T$ with the polynomial
$$
F=\sum_{i=1}^{r-m}c_{i}\omega_if(X)\in\L[X],
$$
which is of degree $d$ as at least one $c_{i}\ne 0$. We get
\begin{equation*}
\left|T\right|\le 2 q^{s-r\vartheta},
\end{equation*}
and thus we conclude the proof.
\end{proof}

We note that $\cI_f(\cA,\cB)>0$ in Theorem~\ref{thm:fBG} if $m>r(1-\vartheta)$.  
Furthermore, for any fixed $\rho>1-\vartheta$ and $m\ge r \rho$, 
Theorem~\ref{thm:fBG} gives an asymptotic formula for $\cI_f(\cA,\cB)$ as $q^r\to\infty$.

When $\cA=\cB$ in Theorem~\ref{thm:fBG}, we get the estimate
$$
\cI_f(\cA,\cA)\le 2q^{s-r\vartheta}.
$$
This bound improves the estimate~\eqref{eq:RNS} obtained in~\cite{RNS} for $$s< 2.5\(\frac{5}{4}\)^dr\varepsilon.$$ In particular, if $\varepsilon=s/r$, it always improves~\eqref{eq:RNS} whenever $d$ satisfies the condition~\eqref{eq:h1}. Moreover, Theorem~\ref{thm:fBG} generalises~\eqref{eq:RNS}  as this estimate was obtained in~\cite{RNS} only for polynomials of degree $d<p$.

Note also that the results of  Roche-Newton and Shparlinski~\cite{RNS}
always required $\eta \ge 1/2$ (but also applies to polynomials of
lower degree).

\begin{cor}
\label{cor:fBG} If under the conditions of Theorem~\ref{thm:fBG} we
have $f(\cA)\subseteq\cB$, then $\cB=\L$.
\end{cor}

\begin{proof}
Indeed, if $f(\cA)\subseteq\cB$, then from  Theorem~\ref{thm:fBG}   we derive
$$
q^s=\cI_f(\cA,\cB) \le q^{s+m-r}+2q^{s-r\vartheta},
$$
which is possible only if $m=r$.
\end{proof}

Theorem~\ref{thm:fBG} has also direct consequences on the image and kernel subspaces of $q$-polynomials defined by
\begin{equation}
\label{eq:linq}
 l=\sum_{i=1}^{\nu}b_iX^{q^i}\in\L[X],\ \nu<r.
\end{equation}
Then, for an affine subspace $\cB$ of $\L$ of dimension $m\le r$, the image set $l(\cB)=\{l(x)\mid x\in\cB\}$ is a $\K$-affine subspace of dimension at most $m$. 

Moreover, we denote by $Ker(l)$ the set of zeroes of the polynomial $l$. By~\cite[Theorem 3.50]{LN}, $Ker(l)$ is a $\K$-linear subspace of $\F_{q^t}$, where $\F_{q^t}$ is the field extension of $\L$ containing all the roots of $l$. Taking now the trace over $\L$, we have that $\Tr_{\F_{q^t}|\L}(Ker(l))$ is a $\K$-linear subspace of $\L$.

Under the conditions of Theorem~\ref{thm:fBG}, for any $q$-poly\-no\-mial $l\in\L[X]$ 
defined by~\eqref{eq:linq}, we have
$$
\left|\cI_f(\cA,l(\cB)) - q^{s+m-r}\right|\le 2q^{s-r\vartheta},
$$
where 
$\vartheta$ is defined by~\eqref{eq:theta}. The same estimate holds for 
$$\cI_f(\cA,\Tr_{\F_{q^t}|\L}(Ker(l)))$$ with $m$ replaced with $\dim_{\K}\Tr_{\F_{q^t}|\L}(Ker(l))$.

Moreover, as in Corollary~\ref{cor:fBG}, we see that
$f(\cA)$ is not included in  $l(\cB)$ for any proper subspace $\cB\subseteq\L$ 
or in  $\Tr_{\F_{q^t}|\L}(Ker(l))$.

It would be certainly interesting to find upper bounds for the intersection of image sets of polynomials on affine subspaces. That is, given $f,g\in\L[X]$, find estimates for the size of $f(\cA)\cap g(\cA)$ for a given proper affine subspace $\cA\subset\L$. For prime fields, Chang shows in~\cite{Chang3} that the intersection of the images of two polynomials on a given interval is sparse. In the case of arbitrary finite fields, several such estimates are given in~\cite{CillShp} for very special classes of polynomials and affine spaces.

\section{Polynomial orbits in subspaces}

As in~\cite{RNS}, one can obtain immediately from Theorem~\ref{thm:fBG} the following consequence about the number of consecutive iterates falling in a subspace. We recall that for a polynomial $f\in\L[X]$ and element $u\in\L$, we define  $T_{f,u} = \# \Orb_{f} (u)$ as defined by~\eqref{eq:Orbu}.

\begin{cor}
\label{cor:consorb}
Let $0<\varepsilon,\eta\le 1$ be arbitrary numbers and let $\gamma_{\eta}$ and 
$\delta(\varepsilon,\eta)$ be defined by~\eqref{eq:gamma} and~\eqref{eq:delta}, respectively.
Let $\cA\subseteq\L$ be an $\eta$-good affine subspace of dimension $s$ over $\K$ with
$$
s\ge \varepsilon r.
$$
Let $d$ 
and and 
$f$ be as in Theorem~\ref{thm:charsumpoly}. 
If for some  $u \in \L$ and an integer $N$ with 
$2 \le N \le T_{f,u}$ 
we have 
$$
f^{(n)}(u)\in \cA, \qquad n =0, \ldots, N-1, 
$$
then 
$$
q^s\ge \frac{1}{2} Nq^{r\vartheta},
$$
where $\vartheta$ is defined by~\eqref{eq:theta}.
\end{cor}
\begin{proof}
The result follows directly from Theorem~\ref{thm:fBG} as
$N\le \cI_f(\cA,\cA)\le 2q^{s-r\vartheta}$.
\end{proof}

\begin{remark}
Similarly to Corollary~\ref{cor:consorb} (replacing $\cA$ with the image space of a linearised polynomial $l$), based on  the discussion after Corollary~\ref{cor:fBG},
one can obtain estimates for the number of consecutive elements in the orbit of a polynomial of the form defined in Theorem~\ref{thm:charsumpoly}, that fall in the orbit of $l$ in any point of $\L$. 
\end{remark}

We also note that the proof of~\cite[Theorem 6]{RNS}, using Theorem~\ref{thm:fBG}, can give information about the number of arbitrary (not necessarily consecutive) iterates falling in a subspace. For the sake of completeness we repeat the argument of ~\cite[Theorem 6]{RNS} for the case of subspaces instead of subfields for which this result has been obtained.

We present our bounds in terms of the the parameter $\rho$ which is a 
frequency of iterates of $f\in\L[X]$ in an affine space, that is, 
$\rho = M/N$, where $M$  is the number of positive integers $n \le N$ 
with $f^{(n)}(u)\in \cA$. Again, we obtain a power improvement over 
the trivial bound 
$q^s\ge \rho N$ (where $s = \dim \cA$). 

\begin{thm}
\label{thm:arbiter}
Let $0<\varepsilon,\eta\le 1$ and let $\gamma_{\eta}$ and 
$\delta(\varepsilon,\eta)$ be defined by~\eqref{eq:gamma} and~\eqref{eq:delta}, respectively. Let $\cA\subseteq\L$ be an $\eta$-good affine subspace of dimension $s\ge \varepsilon r$  over $\K$. 
Let $f\in\L[X]$ be a polynomial of degree $d$ such 
that for $N \le T_{f,u}$ we have $f^{(n)}(u)\in \cA$ for at least $\rho N\ge 2$
values of  $n  =1, \ldots, N$. 
If 
$$
\delta(\varepsilon,\eta)\le d^{2\rho^{-1}} \le  \min\(p, 0.9\log_2\log_2q^r\)+1,
$$
then 
$$
q^{s} \ge  \frac{\rho^2N}{32} q^{r \vartheta_{\rho}},
$$
where  
\begin{equation}
\label{eq:theta rho}
\vartheta_{\rho} =\frac{0.9\varepsilon}{2^{2d^{2/\rho}}}.
\end{equation}
\end{thm}
\begin{proof} We follow exactly the same proof as in~\cite[Theorem 6]{RNS}. 
Let $1\le n_1<\ldots<n_M\le N$ be all values such that $f^{(n_i)}(u)\in \cA$, $i =1, \ldots, M$. 
We denote by $A(h)$ the number of $i=1, \ldots, M-1$ 
with $n_{i+1} - n_i = h$. 
Clearly 
$$
\sum_{h=1}^{N} A(h) = M-1 \mand
\sum_{h=1}^{N} A(h)h = n_{M}- n_1 \le N.
$$
Thus, for any integer $H \ge 1$ we have 
\begin{equation*}
\begin{split}
\sum_{h=1}^{H} & A(h) = M-1 - \sum_{h=H+1}^N A(h) \\
&\ge   M-1 - (H+1)^{-1}\sum_{h=H+1}^N A(h)h 
\ge  M-1 - (H+1)^{-1}N.
\end{split}
\end{equation*}

Hence there exists $k \in \{1, \ldots, H\}$ 
with 
\begin{equation}
\label{eq:A(h) H}
A(k) \ge H^{-1} \(M-1  -  (H+1)^{-1}N\). 
\end{equation}

Let $H=\fl{2\rho^{-1}}\ge 1$.
Then 
$$
H^{-1} \(M-1  -  (H+1)^{-1}N\)\ge \frac{M-1}{2H} \ge \frac{(M-1)^2}{4N} 
$$
and we derive from~\eqref{eq:A(h) H} that 
\begin{equation}
\label{eq:A(h)}
A(k) \ge\frac{(M-1)^2}{4N}=\frac{\rho^2N}{4}\(1-\frac{1}{M}\)^2\ge \frac{\rho^2N}{16}.
\end{equation}
Let $\cJ$ be the set of $j \in \{1, \ldots, M-1\}$ 
with $n_{j+1} - n_j = k$. 
Then we have 
$$
f^{(n_j)}(u)\in \cA 
\mand f^{(n_{j+1})}(u) = f^{(k)}\(f^{(n_j)}(u)\) \in \cA , 
$$
that is
$$
\(f^{(n_j)}(u),f^{(k)}\(f^{(n_j)}(u)\)\)\in\cA\cap f^{(k)}(\cA).
$$
Thus, $A(k)\le \cI_{f^{(k)}}(\cA,\cA)$, and from~\eqref{eq:A(h)} and Theorem~\ref{thm:fBG}, we get
$$
\frac{\rho^2N}{16}\le 2q^{s-r\vartheta_{\rho}},
$$
where $\vartheta_{\rho}$ is defined by~\eqref{eq:theta rho}.
We thus conclude the proof. 
\end{proof}

One can also obtain information on the intersection of orbits of a polynomial $f$ of degree $d<p$ with orbits of a $q$-polynomial $l$ (see also the discussion after Corollary~\ref{cor:fBG}).

\begin{cor}
\label{cor:arbiterlin}
Let $0<\varepsilon,\eta\le 1$ and let $\gamma_{\eta}$ and 
$\delta(\varepsilon,\eta)$ be defined by~\eqref{eq:gamma} and~\eqref{eq:delta}, respectively. Let $f\in\L[X]$ be a polynomial of degree $d$ 
and $l\in\L[X]$ a linearsied polynomial of the form~\eqref{eq:linq} such that $l(\L)$ is an $\eta$-good linear subspace of dimension $s\ge \varepsilon r$  over $\K$. 
Let 
$$
M=\#\(\Orb_f(u)\cap\Orb_l(v)\),
$$
and $\rho=M/\min(T_{f,u},T_{l,v})$ the frequency of intersection of the orbits. If
$$
\delta(\varepsilon,\eta)\le d^{2\rho^{-1}} \le  \min\(p, 0.9\log_2\log_2q^r\)+1,
$$
then 
$$
q^{s} \ge  \frac{\rho^2\min(T_{f,u},T_{l,v})}{32} q^{r \vartheta_{\rho}},
$$
where  $\vartheta_{\rho}$ is defined by~\eqref{eq:theta rho}.
\end{cor}
\begin{proof}
As $\Orb_{l}(v)\subset l(\L)$, 
the proof follows exactly as the proof of Theorem~\ref{thm:arbiter}, but with $\cA$ replaced with $l(\L)$ and $N$ replaced with $\min(T_{f,u},T_{l,v})$.
\end{proof}

\section{Exponential sums over consecutive integers}
\label{sec:exp consec int}

In this section we consider $q=p$. For a positive integer $n\le p^r-1$, we consider the $p$-adic representation
\begin{equation}
\label{eq:padic}
n=n_0+n_1p+\ldots+n_{s-1}p^{s-1},\ 0 \le n_j<p,\ j=0,\ldots,s-1,
\end{equation}
for some $s\le r$.

In this section we fix a basis $\omega_0,\ldots,\omega_{r-1}$  of $\L$ over $\F_p$ and define
\begin{equation}
\label{eq:dig}
\xi_n=\sum_{j=0}^{s-1}n_j\omega_j.
\end{equation}

Let $1\le N\le p^r-1$, $f\in\L[X]$ a polynomial of degree $d$ and $\psi$ an additive character of $\L$. In this section we estimate the exponential sum 
$$
S(N)=\sum_{n\le N} \chi(n)\psi(f(\xi_n)),
$$
where $\chi:\N\to\C$ is a $p$-multiplicative function, that is, it satisfies the condition
$$
\chi\(m+tp^k\)=\chi(m)\chi\(tp^k\)
$$
for all $k\ge 0$, $t\ge 0$ and $0\le m<p^k$. This class of functions, as well as the 
closely related  class of $p$-additive functions have been studied in classical 
works of Gelfond~\cite{Gelf} and Delange~\cite{Del}, 
see also~\cite{Drm,Gr,KaSu}
and references therein for more recent developments.

A  large family of such function can be obtained as  
\begin{equation}
\label{eq:chin}
\chi(n) = \exp\(2 \pi i \sum_{j=0}^{s-1}\alpha_j n_j\), 
\end{equation}
where $\alpha_j$, $j=0,1, \ldots$, is a fixed infinite sequence
of real numbers and  $n$ is given by the $p$-adic representation as in~\eqref{eq:padic},
see also~\cite{HLP} for a more general class. 
In particular taking $\alpha_j = \alpha p^j$ and $\alpha_j = \alpha$  for a real $\alpha$, we
obtain the following two natural examples,
$$
\chi(n) = \exp\(2 \pi i \alpha n\) \mand \chi(n) = \exp\(2 \pi i \alpha \sigma_p(n)\), 
$$
respectively, where $\sigma_p(n)$ is the sum of $p$-ary digits of $n$. 

For simplicity we consider the family~\eqref{eq:chin} in the next result.

\begin{thm}
\label{thm:expN}
Let $0<\varepsilon,\eta\le 1$ be arbitrary numbers   and let $\gamma_{\eta}$ and 
$\delta(\varepsilon,\eta)$ be defined by~\eqref{eq:gamma} and~\eqref{eq:delta}, respectively.  Let $p^{s-1}\le N\le p^s-1$ for some $s\le r$ satisfying
$$
s\ge \varepsilon r,
$$
and assume the linear subspace $\cL_s\subseteq\L$ spanned by $\omega_0,\ldots,\omega_{s-1}$ is $\eta$-good. 
Let $f\in\L[X]$ be a polynomial of the form (i), (ii) or (iii) as defined in Theorem~\ref{thm:charsumpoly} with $d$ satisfying the condition
\begin{equation}
\label{eq:deg s}
\delta\(\varepsilon/2,\eta/2\)+1\le d\le  \min\(p, 0.9\log_2\log_2p^r\)+2,
\end{equation}
and $\psi$ an additive character of $\L$. Let
 $\chi:\N\to\C$ be a $p$-multiplicative function defined by~\eqref{eq:chin}. 
Then
$$
|S(N)|\le (Np)^{1-\eta/4}+2Np^{-r\vartheta_{\eta}/2},
$$ 
where 
\begin{equation}
\label{eq:theta_eta}
\vartheta_{\eta}=\frac{0.9\varepsilon(1-\eta/2)}{2^{2d-2}}.
\end{equation}
\end{thm}
\begin{proof}
Let $K=\lceil s(1-\eta/2)\rceil$ and
$M=p^K\lfloor N/p^K\rfloor-1$. 
Our sum becomes
\begin{equation}
\label{eq:S(N)}
|S(N)|\le |S(M)|+p^K\le |S(M)|+p(Np)^{1-\eta/2}.
\end{equation}
From the definition of $M$, we have that
$$
M=\sum_{i=0}^{K-1}(p-1)p^i+Tp^K,
$$
for some $T\le p^{s-K-1}-1\le Np^{-K}-1$. 

We have
\begin{equation*}
\begin{split}
|S(M)|&=\left|\sum_{m<p^k}\chi(m)\sum_{t\le T}\chi(tp^K)\psi(f(\xi_{m+tp^K}))\right|\\
&\le\sum_{m<p^k}\left|\sum_{t\le T}\chi(tp^K)\psi(f(\xi_{m+tp^K}))\right|
\end{split}
\end{equation*}
and thus, squaring and applying the Cauchy-Schwarz inequality and using the fact that $|\chi(tp^k)|=1$, we get 
\begin{equation*}
\begin{split}
|S(M)|^2&\le p^K\sum_{m<p^K}\left|\sum_{t\le T}\chi(tp^K)\psi(f(\xi_{m+tp^K}))\right|^2\\
&\le p^K\sum_{t_1,t_2\le T} \left|\sum_{m<p^K}\psi(f(\xi_{m+t_1p^K})-f(\xi_{m+t_2p^K}))\right|.
\end{split}
\end{equation*}

The set of integers $n\le M$ is of the form
\begin{equation}
\label{eq:n<M}
\left\{\sum_{i=0}^{K-1}n_ip^i+tp^K\mid 0\le n_0,\ldots,n_{K-1}\le p-1, 0\le t\le T\right\},
\end{equation}
and thus, we now see from~\eqref{eq:dig} that 
$$
\xi_{m+t_ip^K}=\xi_m+\zeta_{t_i}, \quad \xi_m\in\cL_K, \quad i=1,2,
$$
where $\cL_K$ is the $K$-dimensional linear subspace defined by the basis elements $\omega_0,\ldots,\omega_{K-1}$ of $\L$ over $\F_p$, and with some
$\zeta_{t_i}  \in \L$, $0 \le t_i \le T$, $i=1,2$. As $m$ runs over the interval $[0,p^K-1]$, $\xi_m$ runs over all the elements of $\cL_K$, and moreover, $\zeta_{t_1}\ne \zeta_{t_2}$ for $t_1\ne t_2$.

Our sum becomes
\begin{equation*}
\begin{split}
|S(M)|^2&\le p^K\sum_{t_1,t_2\le T} \left|\sum_{x\in\cL_K}\psi(f(x+\zeta_{t_1})-f(x+\zeta_{t_2})\right|\\
&\le Np^K+p^K\sum_{t_1,t_2\le T, t_1\ne t_2} \left|\sum_{x\in\cL_K}\psi(F_{t_1,t_2}(x))\right|,
\end{split}
\end{equation*}
where $F_{t_1,t_2}(X)=f(X+\zeta_{t_1})-f(X+\zeta_{t_2})\in\L[X]$. 

We note that, as $f\in\L[X]$ is a polynomial of the form (i), (ii) or (iii) as defined in Theorem~\ref{thm:charsumpoly}, then $F_{t_1,t_2}$ is a non constant polynomial of the same form as $f$. 
When $f$ is of the form (i), we have $$f=X^{p^{\nu}}g_{\nu}+\ldots+X^pg_1+g_0,$$ where  $g_i\in\L[X]$ , $i=0\ldots,\nu,
$ are such that
$\deg g_0=d\ge \deg g_i+3$, $i=1\ldots,\nu$. Then we get
\begin{equation*}
\begin{split}
F_{t_1,t_2}(X)&=X^{p^{\nu}}\(g_{\nu}(X+\zeta_{t_1})-g_{\nu}(X+\zeta_{t_2})\)+\ldots\\
&\qquad\qquad\qquad+X^p\(g_1(X+\zeta_{t_1})-g_1(X+\zeta_{t_2})\)+F_{0,t_1,t_2}(X),
\end{split}
\end{equation*}
where
\begin{equation*}
\begin{split}
F_{0,t_1,t_2}(X)=g_0(X+\zeta_{t_1})&-g_0(X+\zeta_{t_2})\\
&\quad+\sum_{i=1}^{\nu}\(\zeta_{t_1}^{p^i}g_i(X+\zeta_{t_1})-\zeta_{t_2}^{p^i}g_i(X+\zeta_{t_2})\).
\end{split}
\end{equation*}
For $t_1\ne t_2$, we note that $g_i(X+\zeta_{t_1})-g_i(X+\zeta_{t_2})$, $i=0,\ldots,\nu$, is a nonconstant polynomial of degree equal to $\deg g_i-1$, and 
$$
d-1=\deg F_{0,t_1,t_2}\ge \deg\(g_i(X+\zeta_{t_1})-g_i(X+\zeta_{t_2})\)+3.
$$
Thus, $F_{t_1,t_2}$ is of the same form and satisfies the same conditions as $f$. 

Similarly, if $f$ is of the form (ii) of Theorem~\ref{thm:charsumpoly}, that is $$f=X^{p^{\nu}+p^{\nu-1}+\ldots+p+1}+g\in\L[X],$$ where $g\in\L[X]$ with $\deg g = d\ge 5$, then
$$
F_{t_1,t_2}(X)=g(X+\zeta_{t_1})-g(X+\zeta_{t_2})
$$
is a non constant polynomial of degree $d-1\ge 4$. 

If $f$ is of the form (iii) of Theorem~\ref{thm:charsumpoly}, that is, $f=g(l(x))$ with $\deg g=d$ and some permutation $p$-polynomial $l\in\L[X]$, then
$$
F_{t_1,t_2}(X)=g(l(X)+l(\zeta_{t_1}))-g(l(X)+l(\zeta_{t_2}))=G_{t_1,t_2}(l(X)),
$$
where $G_{t_1,t_2}(X)=g(X+l(\zeta_{t_1}))-g(X+l(\zeta_{t_2}))\in\L[X]$ is of degree $d-1$.

As $s\ge \varepsilon r$, then $K\ge s(1-\eta/2)\ge \varepsilon_{\eta}r$, where $\varepsilon_{\eta}=\varepsilon(1-\eta/2)$ by the hypothesis. Since 
$$
K\ge s(1 - \eta/2) > s\frac{1-\eta}{1-\eta/2}
$$
we also have, for any proper subfield $\F$ of $\L$,
$$
\#\(\cL_K\cap b\F\)\le \#\(\cL_s\cap b\F\)\le p^{s(1-\eta)}<p^{K(1-\eta/2)}.
$$
Moreover, from condition~\eqref{eq:deg s}, we have $d-1\ge\delta(\varepsilon/2,\eta/2)\ge \delta(\varepsilon_{\eta},\eta/2)$ as defined by~\eqref{eq:delta}. 
Thus, the conditions of Theorem~\ref{thm:charsumpoly} are satisfied (with $d$ replaced by $d-1$, $\varepsilon$ replaced by $\varepsilon_{\eta}$ and $\eta$ replaced by $\eta/2$), and we obtain
\begin{equation*}
\begin{split}
|S(M)|^2\le Np^K+2p^KT^2p^{K-r\vartheta_{\eta}}&\le Np^{s(1-\eta/2)+1}+2N^2p^{-r\vartheta_{\eta}}\\
&\qquad\qquad\le (Np)^{2-\eta/2}+2N^2p^{-r\vartheta_{\eta}},
\end{split}
\end{equation*}
where $\vartheta_{\eta}$ is given by~\eqref{eq:theta_eta}
and thus, recalling~\eqref{eq:S(N)}, 
we conclude the proof.
\end{proof}

We also note that 
we have not 
put any efforts in optimising the condition~\eqref{eq:deg s} in Theorem~\ref{thm:expN}. For example, if one imposes the condition
$$
\delta(\varepsilon(1-0.9\eta),0.1\eta)+1\le d\le  \min\(p, 0.9\log_2\log_2p^r\)+2,
$$
then one obtains the slightly better bound
$$
|S(N)|\le (Np)^{1-0.45\eta}+2Np^{-r\vartheta_{\eta}/2},
$$ 
where 
$$
\vartheta_{\eta}=\frac{0.9\varepsilon(1-0.9\eta)}{2^{2d-2}}.
$$

\begin{remark} 
\label{rem:RN}
We note that similarly to the proof of Theorem~\ref{thm:expN} we can derive directly from Theorem~\ref{thm:charsumpoly} a bound for the exponential sum 
$$
R(N)=\sum_{n\le N} \chi(\xi_n)\psi(f(\xi_n)),
$$
where $f\in\L[X]$ is of the form (i), (ii) or (iii) as defined in Theorem~\ref{thm:charsumpoly} with $d$ satisfying the condition
$$
\delta\(\varepsilon/2,\eta/2\)\le d\le  \min\(p, 0.9\log_2\log_2p^r\)+1.
$$
Let  
$\chi:\L\to\C$ satisfy the conditions $\chi(x+y)=\chi(x)\chi(y)$, $x,y\in\L$, and
$$
\sum_{x\in\cA_s}|\chi(x)|^{2^d}\le B.
$$ 
Then, one obtains
$$
|R(N)|\le 2B^{(d+1)/2^{d+1}}Np^{-K(d+1)/2^{d+1}-r\vartheta_{\eta}}+p(Np)^{1-\eta/2}\max_{n\le N}|\chi(\xi_n)|,
$$ 
where 
$$
\vartheta_{\eta}=\frac{0.9\varepsilon(1-\eta/2)}{2^{2d}}.
$$ 

Indeed, as in the proof of Theorem~\ref{thm:expN} we reduce the problem to estimating $|R(M)|$, where $M=p^K\lfloor N/p^K\rfloor-1$. 
As in the proof of Theorem~\ref{thm:expN}, the set of integers $n\le M$ is of the form~\eqref{eq:n<M}, 
and thus, we now see from~\eqref{eq:dig} that the set of $\xi_n$ is partitioned
into the union of $T+1$ affine spaces of the shape
$A(t) = \cL_K + \zeta_t$, where $\cL_K$ is the $K$-dimensional linear subspace defined by the basis elements $\omega_0,\ldots,\omega_{K-1}$ of $\L$ over $\F_p$, and with some
$\zeta_t  \in \L$, $0 \le t \le T$.

As there are at most $N/q^K$ elements $\xi_t\in\L$ corresponding to $t\le T$ 
as discussed above, our sum becomes
\begin{equation*}
\begin{split}
|R(M)|&\le Nq^{-K}\left|\sum_{x\in\cA(t)}\chi(x)\psi(f(x))\right|,
\end{split}
\end{equation*}
where $\cA(t)=\cL_K + \zeta_t$ for some $\zeta_t  \in \L$, $0 \le t \le T$. Now, the estimate follows applying Theorem~\ref{thm:charsumpoly} to the sum $R(M)$.

Moreover, if $N=p^s-1$, for some $s\le r$, the set of elements $\xi_n$ corresponding to $n\le N$ given by~\eqref{eq:dig} defines an affine subspace $\cA$ of $\L$ of dimension $s$. This case is exactly Theorem~\ref{thm:charsumpoly}, and thus
$$
|R(N)|\le 
2B^{(d+1)/2^{d+1}}p^{s\(1-(d+1)/2^{d+1}\)-r\vartheta},
$$
where $\vartheta$ is defined by~\eqref{eq:theta} (but with $d$ replaced by $d-1$).
\end{remark}

\section{Waring problem in intervals and subspaces}

Let $f\in\L[X]$ be a polynomial of degree $d$. 
In this section we consider first  the Waring problem over an affine subspace $\cA$ of $\L$ of dimension $s$, that is the question of the existence and estimation 
of a positive integer $k$ such that, for any $y\in\L$, the equation 
\begin{equation}
\label{eq:War qs}
f(x_1)+\ldots+f(x_k)=y,
\end{equation}
is solvable in $x_1,\ldots,x_k\in\cA$. 

In particular, we denote by $g(f,q,s)$ the smallest possible 
value of $k$ in~\eqref{eq:War qs} and put   $g(f,q,s) = \infty$
if such $k$ does not exist.

We obtain the following direct consequence of Theorem~\ref{thm:charsumpoly}. 

\begin{thm}
\label{thm:warsubsp}
Let $0<\varepsilon,\eta\le 1$ be arbitrary numbers and let $\gamma_{\eta}$ and 
$\delta(\varepsilon,\eta)$ be defined by~\eqref{eq:gamma} and~\eqref{eq:delta}, respectively.
Let $\cA\subseteq\L$ be an $\eta$-good affine subspace of dimension $s$  over $\K$ with
$$
s\ge \varepsilon r.
$$
If $f\in\L[X]$ is a polynomial of the form (i), (ii) or (iii) as defined in 
Theorem~\ref{thm:charsumpoly}, 
then for $k\ge 3$ with
$$
\(\frac{q^{r\vartheta}}{2}\)^{k-2}> D q^{r-s},
$$ 
where $\vartheta$ is defined by~\eqref{eq:theta} and  $D = \deg f$ in the cases (i) and (ii) 
and $D = \deg g$ in the case (iii), 
we have
$$
g(f,q,s)\le k.
$$
\end{thm}

\begin{proof}

We use again exponential sums to count the number of solutions $N_k$ of the equation~\eqref{eq:War qs}, that is,
\begin{equation*}
\begin{split}
N_k&=\frac{1}{q^r}\sum_{u\in\L} \sum_{x_1,\ldots,x_k\in \cA} \psi\(u\(\sum_{i=1}^k f(x_i)-y\)\)
\end{split}
\end{equation*}
and thus
\begin{equation*}
\begin{split}
|N_k-q^{sk-r}|&\le \frac{1}{q^r}\sum_{u\in\L^*} \left|S_u\right|^k=\frac{1}{q^r}\sum_{u\in\L^*} \left|S_u\right|^{k-2} \left|S_u\right|^2\\
&\le \frac{1}{q^r}\sum_{u\in\L^*}\left|S_u\right|^{k-2}\sum_{x_1,x_2\in\cA}\psi(u(f(x_1)-f(x_2))),
\end{split}
\end{equation*}
where
$$
S_u=\sum_{x\in\cA} \psi(uf(x)).
$$
Using Theorem~\ref{thm:charsumpoly} for the sum $S_u$ and the estimate $Dq^s$ (for fixed $x_1\in\cA$, there are at most $D$ zeros of $f(x_1)-f(X)$) for the inner sum, we obtain
$$
\left|N_k-q^{sk-r}\right|\le 2^{k-2}Dq^{s(k-1)-r\vartheta(k-2)},
$$
where $\vartheta$ is defined by~\eqref{eq:theta}. 
Imposing now $N_k> 0$, we conclude the proof.

The statement for the polynomial of the type (iii) in Theorem~\ref{thm:charsumpoly} follows as $l$ is a permutation $p$-polynomial as defined in Theorem~\ref{thm:charsumpoly}.
\end{proof}

If $D$ is fixed in Theorem~\ref{thm:warsubsp}, then
for 
$$
k>\frac{r-s}{r}\vartheta^{-1}+2
$$ 
and sufficiently large $q^s$ we have $g(f,q,s)<k$.

Next we consider $q=p$, and for an integer $n\le N$, we have $\xi_n$ defined by~\eqref{eq:dig}. We also study the question of the existence of a positive integer $k$ such that for any $y\in\L$, the equation 
$$
f(\xi_{n_1})+\ldots+f(\xi_{n_k})=y
$$
is solvable in positive integers $n_1,\ldots,n_k\le N$. 
As above, we denote by $G(f,p,N)$ the smallest such 
value of $k$ and put   $G(f,p,N) = \infty$
if such $k$ does not exist. 

\begin{cor}
\label{cor:War N}
Let $f\in\L[X]$ be a polynomial  of the form (i), (ii) or (iii) as defined in Theorem~\ref{thm:charsumpoly} and $p^{s-1}\le N<p^s$ for some $s\le r$ satisfying
$$
s\ge \varepsilon r,
$$
and assume the linear subspace $\cL_s\subseteq\L$ spanned by $\omega_0,\ldots,\omega_{s-1}$ is $\eta$-good.
 Then
for $k\ge 3$ with
$$
\(\frac{p^{r\vartheta}}{2}\)^{k-2}> D p^{r-s+1},
$$ 
where $\vartheta$ is defined by~\eqref{eq:theta} and  $D = \deg f$ in the cases (i) and (ii) 
and $D = \deg g$ in the case (iii), 
we have
$$
G(f,p,N)\le k.
$$
\end{cor}
\begin{proof}
As $N\ge p^{s-1}$, we have that $G(f,p,N)\le g(f,p,s-1)$, and thus we can apply directly Theorem~\ref{thm:warsubsp} with $s$ replaced with $s-1$, and with $q$ replaced by $p$.
\end{proof}

We note that Corollary~\ref{cor:War N} follows also by applying directly Theorems~\ref{thm:expN}, however the estimate obtained would be slightly weaker.

\section{Remarks and open questions}
We note that we could prove Theorem~\ref{thm:arbiter} only for polynomials of degree less than $p$. The reason behind this is that when one iterates the polynomial $f$ of the form (i), (ii) or (iii), the shape changes and thus we cannot apply anymore Theorem~\ref{thm:charsumpoly}. It would be interesting to extend such a result for more general polynomials. 

Theorem~\ref{thm:fBG}  can also be translated into the language of affine 
dispersers, see~\cite{BsK}. 
We consider $q=p$ prime and $\L=\F_{p^r}$.
\begin{definition}
A function $f:\L\to \F_p$ is an $\F_p$-affine disperser for dimension $s$ if for every affine subspace $\cA$ of $\L$ of dimension at least $s$, we have $\#f(\cA)>1$.
\end{definition}

As a direct consequence of  Theorems~\ref{thm:fBG}, we obtain the following result.
\begin{cor}
Let $0<\varepsilon,\eta\le 1$ 
and let $f\in\L[X]$ be a polynomial 
as defined  in (i), (ii) or (iii) of Theorem~\ref{thm:fBG}.
Then $\pi(f)$, where $\pi:\L\to \F_p$ is a nontrivial $\F_p$-linear map, is an affine disperser for dimension greater than $\varepsilon r$. 
\end{cor}

We note that condition~\eqref{eq:h1} shows that the larger $\varepsilon$ is, the smaller the degree $d$ is, where $d$ is defined as in Theorem~\ref{thm:fBG}. 
For example, if 
$$\varepsilon=\frac{1}{2} \mand \eta \le \frac{4}{5215}
$$
then one has $d>\delta(1/2,\eta)=\gamma_{\eta}^{-1}+3 = 156453$. Furthermore, 
if 
$$\varepsilon=\frac{1}{3} \mand \eta \le  \frac{4}{5215}
$$
then
$d>\delta(1/3,\eta)=2\gamma_{\eta}^{-1}+3 = 312903$.

As mentioned in~\cite{RNS}, obtaining analogues of Theorem~\ref{thm:charsumpoly}, and thus of the rest of results of this paper, for rational functions is an important open direction. For this one has to obtain estimates for the exponential sum
$$
S=\sum_{x\in\cL}\psi\(h(x)\),
$$
where $h\in\L(X)$ is a rational function and $\psi$ a nontrivial additive character.
 Even the case $h(X)=X^{-1}$ 
is still open.

Also of interest is obtaining estimates for
$$
S=\sum_{x\in\cG}\psi\(h(x)\),
$$
where $\cG$ is a multiplicative subgroup of $\L^*$. We note that for the prime field case, such a result would follow from~\cite[Theorem 1]{Bou1}.

Of interest is also the multivariate case of Tehorem~\ref{thm:fBG}, that is, given $F\in\L[X_1,\ldots,X_n]$ and $\cA_1,\ldots,\cA_n,\cB$ affine subspaces of $L$, estimate the size of $F(\cA_1,\ldots,\cA_n)\cap \cB$.

\section*{Acknowledgements}

The author would like to thank Igor Shparlinski for suggesting some extensions of initial results, and for his important comments on earlier versions of the paper. 
The author is also grateful to the Max Planck Institute for Mathematics for hosting the author for two months during the program ``Dynamics and Numbers"  when important progress on this paper was made.

During the preparation of this paper
the author was supported by the UNSW Vice Chancellor's Fellowship.


\begin{thebibliography}{9}


\bibitem{BsK} E. Ben-Sasson and S. Kopparty, `Affine dispersers from subspace polynomials', {\it SIAM
J. Comput.}, {\bf 41} (2012), 880Ð914.
\bibitem{Bou1} J. Bourgain, `Mordell's exponential sum estimate revisited', {\it J. of AMS}, {\bf 18} (2005), 477--499.
\bibitem{Bou} J. Bourgain, `On exponential sums in finite fields',
{\it Bolyai Soc. Math. Stud.}, {\bf 21}, J‡nos Bolyai Math. Soc., Budapest (2010), 219--242.
\bibitem{BG} J. Bourgain and A. Glibichuk, `Exponential sum estimates over a subgroup in an arbitrary finite field', {\it J. D'Analyse Math.}, {\bf 115} (2011), 51--70.

\bibitem{Chang1} M.-C. Chang, `Polynomial iteration in characteristic $p$',
{\it J. Functional Analysis\/},  {\bf  263} (2012), 3412--3421.

\bibitem{Chang2} M.-C. Chang, `Expansions of quadratic maps in prime fields',
{\it Proc. Amer. Math. Soc.\/},  {\bf  142} (2014), 85--92.

 
\bibitem{Chang3} M.-C. Chang, `Sparsity of the intersection of polynomial
images of an interval', {\it Acta Arith.\/}, {\bf 165} (2014), 243--249.

 \bibitem{CCGHSZ} M.-C. Chang, J. Cilleruelo, M. Z. Garaev, J.  Hern\'andez,
I. E. Shparlinski and A. Zumalac\'{a}rregui,
`Points on curves in small boxes and applications',
{\it Michigan Math. J. \/}, {\bf 63} (2014), 503--534.

\bibitem{CGOS} J. Cilleruelo, M. Z. Garaev,  A. Ostafe and
I. E. Shparlinski,
`On the concentration of points of polynomial maps
and applications',
{\it Math. Zeit.\/}, {\bf 272} (2012), 825--837.

\bibitem{CillShp} J. Cilleruelo and I. E. Shparlinski,
`Concentration of points on curves in finite fields',  
{\it  Monatsh. Math.\/},  {\bf 171} (2013), 315--327. 

\bibitem{Cip} J. Cipra, 
`Waring's number in a finite field', 
{\it Integers\/}, {\bf 9} (2009), 435--440. 

\bibitem{CCP} J. Cipra, T. Cochrane and C. G. Pinner, 
`Heilbronn's Conjecture on Waring's number $\bmod\, p$', 
{\it J. Number Theory\/}, {\bf  125} (2007), 289--297.

\bibitem{CoPi}  T. Cochrane and C. Pinner,
`Sum-product estimates applied to Waring's problem   $\bmod\, p$', 
{\it Integers\/}, {\bf 8} (2008),   A46, 1--18. 

\bibitem{Del} H. Delange, `Sur les fonctions $q$-additives ou $q$-multiplicatives', 
{\it Acta Arith.}, {\bf 21}
(1972), 285--298.

\bibitem{Drm} M. Drmota, 
`The joint distribution of $q$-additive functions', 
{\it Acta Arith.\/}, {\bf 100} (2001), no. 1, 17--39. 

\bibitem{Gelf} A. O. Gelfond, `Sur les nombres qui ont des propri\' et\' es additives et multiplicatives donn\' ees, {\it Acta
Arith.}, {\bf 13} (1967/1968), 259--265.

\bibitem{Gr} P. J. Grabner, `Completely $q$-multiplicative functions: 
the Mellin transform approach', {\it Acta Arith.}, {\bf 65}
(1993), 85--96.

\bibitem{GuShp} J. Gutierrez and I.~E.~Shparlinski,
`Expansion of orbits of some dynamical systems over finite fields',
{\it Bull. Aust. Math. Soc.\/}, {\bf  82} (2010),  232--239.


\bibitem{HLP}
R. Hofer, G. Larcher and F. Pillichshammer,
`Average growth-behavior and distribution properties of 
generalized weighted digit-block-counting functions',  
{\it  Monatsh. Math.\/},  {\bf 154} (2008), 99--230.


\bibitem{KaSu} I. K{\'a}tai and M. V. Subbarao, 
`Distribution of additive and q-additive functions under some conditions',  
{\it Publ. Math. Debrecen\/},  {\bf 64} (2004),  167--187. 

\bibitem{LN} R. Lidl and H. Niederreiter, 1997. 
Finite fields, second ed,
Cambridge University Press, Cambridge.

\bibitem{RNS} O. Roche-Newton and I. E. Shparlinski, `Polynomial values in subfields and 
affine subspaces of finite fields', {\it Quart. J. Math.\/}, (to appear).
%\url{http://arxiv.org/abs/1407.2273}.
%(available from {\tt
%http://arxiv.org/abs/1407.2273}).

\bibitem{SiVi}
J. H. Silverman and B. Viray, `On a uniform bound for the number of exceptional linear subvarieties in the dynamical Mordell-Lang conjecture',
{\it Math. Res. Letters.\/},  {\bf  20} (2013),  547--566. 


\bibitem{Weyl} H. Weyl, `\" Uber die Gleichverteilung von Zahlen mod Eins', {\it Math. Ann.}, {\bf 77}
(1916), 313--352.

\bibitem{vsWoWi} 
A. Winterhof and C. van de Woestijne, 
`Exact solutions to Waring's problem in finite fields',
{\it Acta Arith.\/}, {\bf 141} (2010), 171--190.  

\bibitem{Wool1}  T.~D.~Wooley,
`Vinogradov's mean value theorem via efficient congruencing, II',
{\it Duke Math. J.\/}, {\bf 162} (2013), 673--730.

 \bibitem{Wool2} T.~D.~Wooley,
`Translation invariance, exponential sums,
and Waring's problem',  {\it Proc.  Intern. Congress of Math., 
Seoul, 2014\/},  (to appear),  \url{http://arxiv.org/abs/1404.3508}.

 \bibitem{Wool3} T.~D.~Wooley,
`Multigrade efficient congruencing and VinogradovÕs mean value theorem', 
{\it Preprint\/}, 2013, \url{http://arxiv.org/abs/1310.8447}.

\bibitem{WuLi} B. Wu and Z. Liu, `Linearized polynomials over finite fields revisited', {\it Finite Fields Appl.}, {\bf 22} (2013), 79--100.
\end{thebibliography}
\end{document}